\documentclass[a4paper]{article}
\usepackage{amsthm,amssymb,amsmath,enumerate,mathtools} 
\usepackage{graphicx}

\newcommand{\N}{\ensuremath{\mathbb{N}}}

\newcommand{\Z}{\ensuremath{\mathbb{Z}}}
\newcommand{\G}{G=\langle S\rangle}
\newcommand{\bbN}{\ensuremath{\mathbb{N}}}

\newcommand{\Pcal}{\ensuremath{\mathcal{P}}}

\newcommand{\Gammabar}{\ensuremath{\bar{\Gamma}}}
\newcommand{\Aut}{\ensuremath{\mathsf{Aut}}}
\newcommand{\AutG}{\ensuremath{\mathsf{Aut}(\Gamma)}}

\newcommand{\diam}{\ensuremath{\mathsf{diam}}}
\newcommand{\B}{\ensuremath{B^*}}
\newcommand{\C}{\ensuremath{C^*}}
\newcommand{\A}{\ensuremath{A^*}}

\newcommand{\free}{\ast\!\! }

\DeclareMathOperator{\defi}{:\!=}

\theoremstyle{definition}

\theoremstyle{plain}
\newtheorem{theorem}{Theorem}[section]
\newtheorem{lemma}[theorem]{Lemma}

\newtheorem{thm}[theorem]{Theorem}

\newtheorem{coro}[theorem]{Corollary}

\newtheorem{example}[theorem]{Example}

\theoremstyle{remark}

\usepackage[hidelinks=true]{hyperref}

\newenvironment{txteq*}
{
	\begin{equation*}
	\begin{minipage}[t]{0.85\textwidth} 
	\em                                
}
{\end{minipage}\end{equation*}\ignorespacesafterend}

\title{Two-ended quasi-transitive graphs}
\date{\today}
\author{Babak Miraftab \and Tim R\"uhmann \smallskip \and Department of Mathematics\\
 University of Hamburg\\
 Bundesstra\ss e 55\\
 20146 Hamburg\\
 Germany}

\newcommand{\w}{\omega_1}

\begin{document}
 \maketitle

\begin{abstract}
The well-known characterization of two-ended groups says that every two-ended group can be split over  finite subgroups which means it is isomorphic to either by a free product with amalgamation~$A\ast_C B$ or an HNN-extension~$\ast_{\phi} C$, where~$C$ is a finite group and~$[A:C]=[B:C]=2$ and~$\phi\in\Aut(C)$.
In this paper we show that there is a way in order to spilt two-ended quasi-transitive graphs without dominated ends and two-ended transitive graphs over finite subgraphs in the above sense. 
As an application of it, we characterize all groups acting with finitely many orbits almost freely on those graphs.
 \end{abstract}

\section{Introduction}
End theory plays a crucial role in graph theory, topology and group theory, see the work of Diestel, Halin, Hughes, Ranicki, M\"{o}ller and Wall~\cite{sur2,sur1,halin64,RoggiEndsI,RoggiEndsII,wall-geometry}.
In 1931, Freudenthal~\cite{freu31} defined the concept of ends for topological spaces and topological groups for the first time. 
Let~$X$ be a locally compact Hausdorff space.
In order to define ends of the topological space~$X$, consider infinite sequence~$U_1\supseteq U_2 \supseteq \cdots$ of non-empty connected open subsets of~$X$ such that the boundary of each~$U_i$ is compact and~$\bigcap\overline{U_i}=\emptyset$.
Two sequences~$U_1\supseteq U_2 \supseteq \cdots$ and~$V_1\supseteq V_2 \supseteq \cdots$ are \emph{equivalent} if for every~${i\in\N}$, there are~$j,k\in\N$ in such a way that~$U_i\supseteq V_j$ and~$V_i\supseteq U_k$.
The equivalence classes of those sequences are the \emph{ends} of~$X$.
The ends of groups arose from ends of topological spaces in the work of Hopf~\cite{hopf}.
In 1964, Halin~\cite{halin64}   defined ends(vertex-ends) for infinite  graphs independently as equivalence classes of rays, one way infinite paths.
Diestel and K\"{u}hn~\cite{ends} showed that if we consider locally finite graphs as one dimensional 
simplicial complexes, then these two concepts coincide. 
We can define the number of ends for a given finitely generated group~$G$ as the number of ends of a Cayley graph of~$G$.	
It is known that the number of ends of two Cayley graphs of the same group are equal, as long as the generating sets are finite, see~\cite{mei}.\footnote{Even more stronger, they are quasi-isometric.}
Freudenthal~\cite{freu44} and Hopf~\cite{hopf} proved that the number of ends for infinite groups~$G$ is either 1, 2 or~$\infty$.
Subsequently  Diestel, Jung and M\"oller~\cite{DiestelJungMoeller} extended the above result to arbitrary (not necessarily locally finite) transitive graphs.
They proved that the number of ends of an arbitrary infinite connected transitive graph is either 1,2 or~$\infty$.
In 1943, Hopf~\cite{hopf} characterized two-ended finitely generated groups.
Later, Scott and Wall~\cite{ScottWall} gave another characterization of  two-ended finitely generated groups.
We summarize all of them as the following theorem:
\begin{thm}\label{Classifiication}
	Let~$G$ be a finitely generated group. Then the following statements are equivalent:
	\begin{enumerate}[\rm (i)]
		\item~$G$ is a two-ended group.
		\item Any Cayley graph of~$G\sim_{QI} \Gamma(\Z,\{\pm 1\})$. 
		\item~$G$ has an infinite cyclic subgroup of finite index.
		\item~$G$ is isomorphic to either~$A {{\free\,\,}_{C}}B$ and~$C$ is finite and \\$[A:C]=[B:C]=2$ or~$\ast_\phi C$ with~$C$ is finite and~$\phi \in \Aut( C)$. 
	\end{enumerate} 
\end{thm}	

Our aim is to extend the above theorem for quasi-transitive graphs.
The first obstacle is the free product with amalgamations and HNN-extensions, as they are group theoretical notions.
It turns out that tree-amalgamation  is a good approach to generalize the above theorem to two-ended graphs.
In particular it seems that two-ended graphs split over  finite subgraphs via tree-amalgamations.
Indeed we will show that every quasi-transitive graph without dominated end can be expressed as a tree-amalgamation of two rayless graphs, see Theorem \ref{char-two-ended-graph}.
In particular if the graph is a locally finite graph, then it is a tree-amalgamation of two finite graphs in an analogous manner with Theorem \ref{Classifiication}.\\
In 1984, Jung and Watkins~\cite{jung_watkins}  studied groups acting on two-ended transitive graphs.
In this paper, we also generalize the results mentioned above to two-ended quasi-transitive graphs without dominated ends.


\section{Preliminaries}

We refer the readers to~\cite{diestelBook10noEE}  for the notations and the terminologies of graph-theoretical terms and to~\cite{harpe-book} for combinatorial group-theoretical notations.\\
In the following we will recall the most important definitions and notations for the readers convenience. 

\subsection{Graph theory}
Let~$\Gamma$ be a graph with vertex set~$V$ and edge set~$E$.
For a set~$X \subseteq V$ we set~$\Gamma[X]$ to be the induced subgraph of~$\Gamma$ on~$X$. 
A \emph{ray} is a one-way infinite path in a graph, the infinite sub-paths of a ray are its \emph{tails}.
An \emph{end} of a graph is an equivalence class of rays in a graph in which two rays are equivalent if and only if there exists no finite vertex set~$S$ such that after deleting~$S$ those rays have tails completely contained in different components. 
A sequence of finite vertex sets~$(F_i) _{i \in \bbN}$ is a \emph{defining sequence} of an end~$\omega$ if~$C_{i+1} \subsetneq C_{i}$, with~${C_i \defi C(F_i,\omega)}$ and~$\bigcap C_i = \emptyset$. 
We define the \emph{degree of an end~$\omega$} as the supremum over the number of edge-disjoint rays belonging to the class which corresponds to~$\omega$.
We say an end~$\omega$ \emph{lives} in a component~$C$ of~$\Gamma \setminus X$, where~$X$ is a subset of~$V(\Gamma)$ or a subset of~$E(\Gamma)$, when a ray of~$\omega$ has a tail completely contained in~$C$, and we denote~$C$ by~$C(X,\omega)$.
We say a component of a graph is \emph{big} if there is an end which lives in that component. 
Components which are not big are called \emph{small}.
We define~$s(\Gamma)$ to be the maximum number of disjoint double rays in the graph~$\Gamma$. 
An end~$\omega$ of a graph~$\Gamma$ is \emph{dominated by a vertex}~$v$ if there is no finite sets of vertices~${S \setminus v}$ such that~$v \notin C(S,\omega)\cup S$. 
Note that this implies that~$v$ has infinite degree.
An end is said to be~\emph{dominated} if there exists a vertex dominating it. 
%
A finite set~$C \subseteq E$ is a \emph{finite cut} if there exists a partition~$(A,A^\ast)$ of~$V$ such that~$C$ are exactly the edges between~$(A,A^\ast)$, which we denote by~$C=E(A,A^\ast)$.
A cut~$C=E(A,A^\ast)$ is the cut \emph{induced} by the partition~$(A,A^\ast)$.
We note that if~$C=E(A,A^\ast)$ is a cut, then the partition~$(gA,gA^\ast)$ induces a cut for every~$g\in \mathsf{Aut}(\Gamma)$.
For the sake of simplicity we denote this new cut only by~$gC$.
A finite cut~$C=E(A,A^\ast)$ is called \emph{tight} if $G[A]$ and $G[\A]$ are connected and moreover if $|E(A,\A)|=k$, then we say that $C$ is $k$-\emph{tight}.\\
A concept similar to cuts is the concept of separations. 
A \emph{separation} is a pair~${(A,A^\ast)}$ with~${A,A^\ast \subseteq V}$ such that~$\Gamma=\Gamma[A]\cup \Gamma[A^\ast]$.
The set~$A \cap A^\ast$ is called the \emph{separator} of this separation. 
The \emph{order} of a separation is the size of its separator. 
In this paper we only consider separations of finite order, thus from here on, any separation will always be a separation of finite order. 
%
For two-ended graphs we call a separation~$\mbox{\emph{$k$-tight}}$ if the following holds: 
\begin{enumerate}
	\item~$|A \cap A^*| = k$.
	\item There is an end~$\omega_A$ living in a component~$C_A$ of~$A \setminus A^\ast$. 
	\item There is an end~$\omega_{A^\ast}$ living in a component~$C_A^\ast$ of~$A^\ast \setminus A$.
	\item Each vertex in~$A \cap A^\ast$ is adjacent to vertices in both~$C_A$ and~$C_{A^\ast}$.  
\end{enumerate}
If a separation~$(A,A^\ast)$ is~$k$-tight for some~$k$ then this separation is just called \emph{tight}.
Note that finding tight separations is  always possible for two-ended graphs. 
In an analogous matter to finite cuts, one may see that~$(gA,gA^\ast)$ is a tight separation for~$g\in \mathsf{Aut}(\Gamma)$ whenever~$(A,A^\ast)$ is a tight separation.
Assume that~${(A,A^\ast)}$ and~${(B,B^\ast)}$ are two separations of~$\Gamma$.
We say~${(A,A^\ast) \leq (B,B^\ast)}$ if and only if~$A \subseteq B$ and~$A^\ast \supseteq B^\ast$.
Let us call~${(A, A^\ast)}$ and~${(B, B^\ast)}$ \emph{nested} if~$(A, A^\ast)$ is comparable with~$(B,B^\ast)$ or with~${(B^\ast,B)}$ under~$\leq$.
A separation~$(A,A^\ast)$ is \emph{connected} if~$\Gamma[A \cap A^\ast]$ is connected. 
%
%
%
Next we recall the definition of the \emph{tree-amalgamation} for graphs which was first defined by Mohar~\cite{Mohar06}. 
We use the tree-amalgamation to obtain a generalization of factoring quasi-transitive graphs in a similar manner to  HNN-extensions or free-products with amalgamation over groups.
A tree~$T$ is~$(p_1,p_2)$\emph{-semiregular }if there exist~$p_1,p_2 \in \{1,2,\ldots \} \cup \infty$ such that for the canonical bipartition~$\{V_1,V_2\}$ of~$V(T)$ the vertices in~$V_i$ all have degree~$p_i$ for~$i=1,2$. 
In the following let~$T$ be the~$(p_1,p_2)$-semiregular tree.
Suppose that there is a mapping~$c$ which assigns to each edge of~$T$ a pair~$(k,\ell),\, 0 \leq k < p_1,\, 0 \leq \ell <  p_2$, such that for every vertex~$v \in V_1$, all the first coordiantes of the pairs in~{$\{c(e) \mid v$ is incident with~$e\}$} are distinct and take all values in the set~$\{k \mid 0 \leq k < p_1\}$, and for every vertex in~$V_2$, all the second coordiantes are distinct and exhaust all values of the set~$\{\ell \mid 0 \leq \ell < p_2 \}$.
Let~$\Gamma_1$ and~$\Gamma_2$ be graphs. 
Suppose that~$\{S_k \mid 0 \leq k < p_1\}$ is a family of subsets of~$V(\Gamma_1)$, and~${\{ T_\ell \mid 0 \leq \ell < p_2\}}$  is a family of subsets of~${V(\Gamma_2)}$.
We shall assume that all sets~$S_k$ and~$T_\ell$ have the same cardinality, and we let~${\phi_{k\ell}\colon S_k \rightarrow T_\ell}$ be a bijection. 
The maps~$\phi_{k\ell}$ are called  \emph{identifying maps}.
For each vertex~$v \in V_i$, take a copy~$\Gamma_i^v$ of the graph~${\Gamma_i , i = 1, 2}$. 
Denote by~${S_k^v}$ (if~${i = 1}$) and~${T^v_\ell}$ (if~${i = 2}$) the corresponding copies of~$S_k$ or~$T_\ell$ in~${V(\Gamma^v_i)}$. 
Let us take the disjoint union of graphs~${\Gamma^v_i , v \in V_i , i = 1, 2}$. 
For every edge~${st \in E(T ) (s \in V_1, t \in V_2)}$ with~${c(st) = (k, \ell)}$ we identify each vertex~${x \in S^s_k}$
with the vertex~$y = \phi_{k\ell}(x)$ in~$T^t_\ell$. 
The resulting graph~$Y$ is called the \emph{tree-amalgamation} of the graphs~$\Gamma_1$ and~$\Gamma_2$ over the \emph{connecting tree}~$T$.
We denote~$Y$ by~$\Gamma_1\free_T \Gamma_2$.
In the context of tree-amalgamations the sets~$\{S_k\}$ and~$\{T_\ell\}$ are also called \emph{the sets of adhesions} and a single~$S_k$ or~$T_\ell$ might be called an \emph{adhesions} of this tree-amalgamation. 
In the case that~$\Gamma_1 = \Gamma_2$ and that~$\phi_{k \ell}$ is the identity for all~$k$ and~$\ell$ we may say that~$\{S_k\}$ is the set of adhesions of this tree-amalgamation.
A tree-amalgamation~$\Gamma_1\ast_T\Gamma_2$ is called \emph{thin} if all adhesions are finite and~$T$ is the double ray and moreover if~$\Gamma_1$ and~$\Gamma_2$ are layless, then we call it \emph{strongly thin}.

\subsection{Combinatorial group theory}

Let a group~$G$ act on a set~$X$.
By~$\mathsf{St}_G(x)$, we denote the stabilizer of~$x\in X$, i.e the set of all elements of~$G$ fixing~$x$.
If~$\mathsf{St}_G(x)$ is finite for all $x\in X$, we say that $G$ acts \emph{almost freely} on $X$.

Let~$(X,d_X)$ and~$(Y,d_Y)$ be two metric spaces and let~$\phi\colon X\to Y$ be a map.
The map~$\phi$ is a \emph{quasi-isometric embedding} if there is a constant~$\lambda\geq 1$ such that for all~$x,x^\prime \in X$:
$$\frac{1}{\lambda}d_X(x,x')-\lambda\leq d_Y(\phi(x),\phi(x'))\leq{\lambda}d_X(x,x')+\lambda$$
The map~$\phi$ is called \emph{quasi-dense} if there is a~$\lambda'$ such that for every~$y\in Y$ there exists~$x\in X$ such that~$d_Y(\phi(x),y)\leq \lambda'$.
Finally~$\phi$ is called a \emph{quasi-isometry} if it is both quasi-dense and quasi-isometric embedding.
If~$X$ is quasi-isometric to~$Y$, then we write~$X\sim_{QI} Y$.\\
Remember that~$\G$ can be equipped by the word metric induced by~$S$.
Thus any group can be turned to a topological space by considering its Cayley graph and so we are able to talk about quasi-isometric groups and it would not be ambiguous if we use the notation~${G\sim_{QI} H}$ for two groups~$H$ and~$G$.
Now we have the following important lemma  which reveals the connection between Cayley graphs of a group with different generating sets.

\begin{lemma}{\rm\cite[Theorem 11.37]{mei}}\label{diff generators}
	Let~$G$ be a finitely generated group and let~$S$ and~$T$ be two finite generating sets. 
	Then~$\Gamma(G,S)\sim_{QI}\Gamma(G,T)$.
\end{lemma}

By Lemma~\ref{diff generators} we know that any two Cayley graphs of the same group are quasi-isometric if the corresponding generating sets are finite.   
Let~$\G$ be a finitely generated group.
Brick~\cite{Brick} studied the connection  of quasi-isometric groups and their end spaces.
He proved the following important lemma.
\begin{lemma}{\rm\cite[Corollary 2.3]{Brick}}
	\label{endsequall}
	Finitely generated quasi-isometric groups all have the same number of ends. 
\end{lemma}	 
\begin{coro}{\rm\cite[Theorem 11.23]{mei}}
	\label{same ends}
	The number of ends of a group~$G$ is independent of choosing generating set.
\end{coro}

Next we review the definition of the free product with amalgamation and the HNN-extension.
Let~$G_i$ be three groups such that there are monomorphisms~$\phi_i\colon G_2\to G_i$ for~${i=1,3}$.
Then we denote a \emph{free product with amalgamation}~$G_1$ and~$G_3$ over~$G_2$ and an HNN\emph{-extension} over~$G_2$ by~$G_1\ast_{G_2} G_3$ and~$\ast_{\phi_1}G_1$, respectively.
Finally for a subset~$A$ of a set~$X$ we denote the complement of~$A$ by~$A^c$.
We denote the disjoint union of two sets~$A$ and~$B$ by~$A \sqcup B$.  
%


\section{Characterization of two-ended graphs}

\begin{thm}
	\label{char-two-ended-graph}
	Let~$\Gamma$ be a connected quasi-transitive graph without dominated ends.
	Then the following statements are equivalent:
	\begin{enumerate}[\rm (i)]
		\item~$\Gamma$ is two-ended.
		\item~$\Gamma$ can be split as a strongly thin tree-amalgamation~$\Gammabar \ast_T \Gammabar$ fulfills the following properties: 
		\begin{enumerate}[\rm a)]
			\item~$\Gammabar$ is a connected rayless graph of finite diameter.
			\item The identification maps are all the identity. 
			\item All adhesions of the tree-amalgamation contained in ${\overline{\Gamma}}$ are finite and connected and pairwise disjoint.

		\end{enumerate}	 
		\item~$\Gamma\sim_{QI}$  the double ray. 
	\end{enumerate}
\end{thm}

In Theorem~\ref{char-two-ended-graph} we characterize graphs which are quasi-isometric to the double ray. 
It is worth mentioning that Kr\"{o}n and M\"{o}ller \cite{kron2008quasi} have studied arbitrary graphs which are quasi-isometric to trees.

Before we can prove Theorem~\ref{char-two-ended-graph} we have to collect some tools used in its proof.
The first tool is the following Lemma~\ref{good separation} which basically states that in a two-ended quasi-transitive graph~$\Gamma$ we can find a separation  fulfilling some nice properties. 
For that let us define a \emph{type 1 separation} of~$\Gamma$ as a separation~${(A,\A)}$ of~$\Gamma$ fulfilling the following conditions:
\begin{enumerate}[\rm (i)]
	
	\item~$A \cap A^\ast$ contains an element from each orbit.
	\item~$\Gamma[A \cap A^\ast]$ is a finite connected subgraph. 
	\item Exactly one component of~$A \setminus A^\ast$ is big. 
\end{enumerate}

\begin{lemma}
	\label{good separation}
	Let~$\Gamma$ be a connected two-ended quasi-transitive graph. 
	Then there exists a type 1 separation  of~$\Gamma$. 
\end{lemma}

\begin{proof}
	As the two ends of~$\Gamma$ are not equivalent, there is a finite~$S$ such that the ends of~$\Gamma$ live in different components of~$\Gamma \setminus S$. 
	Let~$C$ be a  big component of~$\Gamma \setminus S$. 
	We set~$\bar{A} \defi C \cup S$ and~$\bar{A}^\ast \defi \Gamma \setminus C$ and  obtain a separation~$(\bar{A}, \bar{A}^\ast)$ fulfilling the condition  (iii). 
	Because~${\bar{A} \cap \bar{A}^\ast = S}$ is finite, we only need to add finitely many finite paths to~$\bar{A} \cap \bar{A}^*$ to connect~${\Gamma[\bar{A} \cap \bar{A}^\ast]}$. 
	As~$\Gamma$ is quasi-transitive there are only finitely many orbits of the action of~$\mathsf{Aut}(\Gamma)$ on~$V(\Gamma)$.
	Picking a vertex from each orbit and a path from that vertex to~$\bar{A} \cap \bar{A}^\ast$ yields a separation~$(A,A^\ast)$ fulfilling all the above listed conditions. 
\end{proof}

In the proof of Lemma~\ref{good separation} we start by picking an arbitrary separation which we then extend to obtain type 1 separation. 
The same process can be used when we start with a tight separation, which yields the following corollary: 

\begin{coro}
	\label{type 2 separation}
	Let~$\Gamma$ be a two-ended quasi-transitive graph and let~$(\bar{A},\bar{A}^*)$ be a tight separation of~$\Gamma$.
	Then there is an extension of~$(\bar{A},\bar{A}^*)$ to a type 1 separation~$(A,A^*)$ such that~$\bar{A} \cap \bar{A}^* \subseteq A \cap A^*$. \qed
\end{coro}

Every separation~$(A,A^*)$ which can be obtained by Corollary~\ref{type 2 separation} is a \emph{type 2 separation}.
We also say that the tight separation~$(\bar{A},\bar{A}^*)$ induces the type 2 separation~$(A,A^*)$.

In Lemma~\ref{far away} we prove that in a quasi-transitive graph without dominated ends there are vertices which have arbitrarily large distances from one another.
This is very useful as it allows to map separators of type 1 separations far enough into big components, such that the image and the preimage of that separation are disjoint.

\begin{lemma}
	\label{far away}
	Let~$\Gamma$ be a connected  two-ended quasi-transitive graph without dominated ends, and let~$(A,A^\ast)$ be a type 1 separation.
	Then for every~$k \in \N$ there is a vertex in each  big component of~$\Gamma \setminus (A \cap A^\ast)$ that has distance at least~$k$ from~$A\cap A^\ast$. 
\end{lemma}

\begin{proof}
	Let~$\Gamma$ and~$(A,A^\ast)$ be given and set~$S \defi A \cap A^\ast$. 
	Additionally let~$\omega$ be an end of~$\Gamma$ and set~$C\defi C(S,\omega)$.
	For a contradiction let us assume that there is a~$k \in \bbN$ such that every vertex of~$C$ has distance at most~$k$ from~$S$. 
	Let~$R=r_1,r_2, \ldots$ be a ray belonging to~$\omega$. 
	We now define a forest~$T$ as a sequence of forests~$T_i$. 
	Let~$T_1$ be a path from~$r_1$ to~$S$ realizing the distance of~$r_1$ and~$S$, i.e.~$T_1$ is a shortest path between~$r_1$ and~$S$. 
	Assume that~$T_i$ is defined. 
	To define~$T_{i+1}$ we start in the vertex~$r_{i+1}$ and follow a shortest path from~$r_{i+1}$ to~$S$.
	Either this path meets a vertex contained in~$T_{i}$, say~$v_{i+1}$, or it does not meet any vertex contained in~$T_{i}$. 
	In the first case let~$P_{i+1}$ be the path from~$r_{i+1}$ to~$v_{i+1}$.
	In the second case we take the entire path as~$P_{i+1}$.
	Set~$T_{i+1} \defi~T_i \cup P_{i+1}$. 
	Note that all~$T_i$ are forests by construction.
	For a vertex~$v \in T_i$ let~$d_i(v,S)$ be the length of a shortest path in~$T_i$ from~$v$ to any vertex in~$S$. 
	Note that as each component of each~$T_i$ contains at exactly one vertex of~$S$ by construction, this is always well-defined. 
	Let~$P=r_i,x_1,x_2,\ldots,x_n,s$ with~$s \in S$ be a shortest path between~$r_i$ and~$S$. 
	As~$P$ is a shortest path between~$r_i$ and~$S$ the subpath of~$P$ starting in~$x_j$ and going to~$s$ is a shortest~$x_j-s$ path. 
	This implies that for~$v$ of any~$T_i$ we have~$d_i(v,S)\leq k$. 
	We now conclude that the diameter of all components of~$T_i$ is at most~$2k$ and hence each component of~${T \defi \bigcup T_i}$ also has diameter at most~$2k$, furthermore note that~$T$ is a forest. 
	As~$S$ is finite there is an infinite component of~$T$, say~$T^\prime$. 
	As~$T^\prime$ is an infinite tree of bounded diameter it contains a vertex of infinite degree, say~$u$. 
	So there are infinitely many paths from~$u$ to~$R$ which only meet in~$u$. 
	But this implies that~$u$ is dominating the ray~$R$, a contradiction.
\end{proof}

Our next tool used in the proof of Theorem~\ref{char-two-ended-graph} is Lemma~\ref{smalldiameter} which basically states that small components have  small diameter. 

\begin{lemma}
	\label{smalldiameter}
	Let~$\Gamma$ be a connected two-ended quasi-transitive graphs without dominated ends.
	Additionally let~$S=S_1 \cup S_2$ be a finite vertex set such that the following holds:
	\begin{enumerate}[\rm (i)]
		\item~$S_1 \cap S_2 = \emptyset$.
		\item~$\Gamma[S_i]$ is connected for~$i=1,2$. 
		\item~$S_i$ contains an element from of each orbit for~$i=1,2$.
		
	\end{enumerate} 
	Let~$H$ be a rayless component of~$\Gamma \setminus S$. 
	Then~$H$ has  finite diameter. 
\end{lemma}

\begin{proof}
	Let~$\Gamma,S$ and~$H$ be given. 
	Assume for a contradiction that~$H$ has unbounded diameter. 
	We are going to find a ray inside of~$H$ to obtain a contradiction.
	Our first aim is to find a~$g \in \Aut(\Gamma)$ such that the following holds:
	\begin{enumerate}[(i)]
		\item~$gS_i \subsetneq H$ 
		\item~$gH \subsetneq H$. 
	\end{enumerate}
	Let~$d_m$ be the maximal diameter of the~$S_i$, and let~$d_d$ be the distance between~$S_1$ and~$S_2$.
	Finally let~$d_S = d_d+2d_m$.
	
	First assume that~$H$ only has neighbors in exactly one~$S_i$.
	This implies that~$\Gamma \setminus H$ is connected. 
	Let~$w$ be a vertex in~$H$ of distance greater than~$2d_S$ from~$S$ and let~$g \in \mathsf{Aut}(\Gamma)$ such that~$w\in gS$. 
	This implies that~$gS \subsetneq H$.
	But as~$\Gamma \setminus H$ contains a ray, we can conclude that~$gH \subsetneq H$.
	Otherwise~$gH$ would contain a ray, as~$\Gamma \setminus H$ contains a ray and is connected. 
	
	So let us now assume that~$H$ has  a neighbor in both~$S_i$. 
	Let~$P$ be a shortest~$S_1-S_2$ path contained in~$H \bigcup (S_1 \cup S_2)$, say~$P$ has length~$k$.  
	We pick a vertex~$w \in H$ of distance at least~$2d_S +k+1$ from~$S$, and we pick a~$g \in \Aut(\Gamma)$ such that~$w \in gS$. 
	Obviously we know that~$gP \subseteq (gH \cup gS)$. 
	By the choice of~$g$ we also know that~$gP \subseteq H$. 
	This yields that~$gH \subseteq H$, as~$gH$ is small.  
	We can conclude that~$gH \neq H$ and hence~$gS_i \subsetneq H$ follows directly by our choice of~$g$. 
	
	Note that as $gH$ is a component of $\Gamma \setminus gS$ fulfilling all conditions we had
	on $H$ we can iterate the above defined process with $gH$ instead of $H$.
	We can now pick a vertex $v\in S$. 
	Let $U$ be the images of $v$. 
	As $H$ is connected we apply the Star-Comb lemma, see \cite[Lemma 8.2.2.]{diestelBook10noEE}, to $H$ and $U$.
	We now show, that the
	result of the Star-Comb lemma cannot be a star. So assume that we obtain
	a star with center $x$. Let $\ell:= |S|$.
	Let $d_X$ be the distance from $S$ to $x$.
	By our construction we know that there is a step in which we use a $g_x\in \Aut(G)$
	such that $d(S, g_xS) > d_x$. 
	Now pick $\ell+ 1$ many leaves of the star which come from steps in the process after we used $g_x$. 
	This implies that in the star, all the paths from those $\ell+1$ many leaves to $x$ have to path through a separator
	of size $\ell$, which is a contradiction. So the Star-Comb lemma yields a comb
	and hence a ray.

\end{proof}

\begin{lemma}
	\label{element infinite order}
	Let~$\Gamma$ be a two-ended connected quasi-transitive graph without dominated ends and let~$(A,A^\ast)$ be a type 1 separation and let~$C$ be the big component of~$A \setminus A^\ast$. 
	Then there is a~$g \in \AutG$ such that~$g(C) \subsetneq C$. 
\end{lemma}

\begin{proof}
	Let~$\Gamma$ be a two-ended connected quasi-transitive graph without dominated ends and let~$(A,A^\ast)$ be a type 1 separation of~$\Gamma$.
	Set~${d \defi \diam(A \cap A^\ast)}$. 
	Say the ends of~$\Gamma$ are~$\omega_1$ and~$\omega_2$ and set~${C_i \defi C(A \cap A^*,\omega_i)}$.
	Our goal now is to find an automorphism~$g$ such that~$g(C_1) \subsetneq C_1$. 
	
	To find the desired automorphism~$g$ first pick a vertex~$v$ of distance~${d+1}$ from~${A \cap A^\ast}$ in~$C_1$. 
	As~$(A,A^\ast)$ is a type 1 separation of the quasi-transitive graph~$\Gamma$ there is an automorphism~$h$ of~$\Gamma$ that maps a vertex of~$A \cap A^\ast$ to~$v$. 
	Because~${\Gamma[A \cap A^\ast]}$ is connected and because~${d(v,A \cap A^\ast) \geq d+1}$ we can conclude that~$(A \cap A^\ast)$ and~${h (A \cap A^\ast)}$ are disjoint. 
	If~$h(C_1) \subsetneq C_1$ we can choose~$g$ to be~$h$, so let us assume that~$h(C_1) \supseteq C_2$. 
	Now pick a vertex~$w$ in~$C_1$ of distance at least~$3d+1$ from~$A \cap A^\ast$, which is again possible by Lemma~\ref{far away}. 
	Let~$f$ be an automorphism such that~${w \in f(A \cap A^\ast)}$. 
	Because~${d(w,A \cap A^\ast) \geq 3d+1}$ we can conclude that~
	$${A \cap A^\ast, ~h(A\cap A^\ast)} \mbox{ and } {f(A \cap A^\ast)}$$ are pairwise disjoint and hence in particular~${f \neq h}$. 
	Again if~$f(C_1) \subsetneq C_1$ we may pick~$f$ as the desired~$g$, so assume that~$f(C_1) \supseteq C_2$.
	
	This implies in particular that~$fC_2 \subsetneq hC_2$ which yields that 
	$$h^{-1}f(C_2) \subsetneq C_2$$ which concludes this proof. 
\end{proof}

Note that the automorphism in Lemma~\ref{element infinite order} has infinite order. 
Now we are ready to prove Theorem~\ref{char-two-ended-graph}. 

\begin{proof}[\rm  \bf Proof of Theorem~\ref{char-two-ended-graph}]
	We start with {\bf (i)~$\Rightarrow$ (ii)}.
	
	\noindent So let~$\Gamma$ be a graph fulfilling the conditions in Theorem~\ref{char-two-ended-graph} and let~$\Gamma$ be two-ended.
	Additionally let~$(A,A^\ast)$ be a type 1 separation of~$\Gamma$ given by Lemma~\ref{good separation} and let~$d$ be the diameter of~$\Gamma[A \cap A^\ast]$. 
	Say the ends of~$\Gamma$ are~$\omega_1$ and~$\omega_2$ and set~${C_i \defi C(A\cap A^*, \omega_i)}$.
	By Lemma~\ref{element infinite order} we know that there is an element~${g \in \AutG}$ such that~$g(C_1) \subsetneq C_1$.

	We know that either~$A \cap gA^*$ or~$\A \cap gA$ is not empty, without loss of generality let us assume the first case happens. 
	Now we are ready to define the desired tree-amalgamation. 
	We define the two graphs~$\Gamma_1$ and~$\Gamma_2$  like follows:

	\begin{align*}
	\Gamma_1 \defi \Gamma_2 \defi  \Gamma[A^\ast \cap g A].
	\end{align*}
	Note that as~$A \cap A^\ast$ is finite and because any vertex of any ray in~$\Gamma$ with distance greater than~$3d+1$ from~${A\cap A^\ast}$ is not contained in~$\Gamma_i$ we can conclude~$\Gamma_i$ is a rayless graph.\footnote{Here we use that any ray belongs to an end in the following manner: Since~$A \cap A^*$ and~${g(A \cap A^*)}$ are finite separator  of~$\Gamma$ separating~$\Gamma_1$ from any~$C_i$, no ray in~$\Gamma_i$ can be equivalent to any ray in any~$C_i$ and hence~$\Gamma$ would contain at least three ends.}
	The tree~$T$ for the tree-amalgamation is just a double ray. 
	The families of subsets of~$V(\Gamma_i)$ are just~$A \cap A^\ast$ and~$g(A\cap A^\ast)$ and the identifying maps are the identity. 
	It is straightforward to check that this indeed defines the desired tree-amalgamation.
	The only thing remaining is to check that~$\Gamma_i$ is connected and has finite diameter. 
	It follows straight from the construction and the fact that~$\Gamma$ is connected that~$\Gamma_i$ is indeed connected.
	
	It remains to show that~$\Gamma_i$ has finite diameter. 
	We can conclude this from  Lemma~\ref{smalldiameter} by setting~$S \defi g^{-1}(A \cap A^\ast) \bigcup g^2 (A\cap A^\ast)$.
	As~$\Gamma_i$ is now contained in a rayless component of~$\Gamma \setminus S$.

	\vspace*{0,5cm}
\noindent {\bf (ii)~$\Rightarrow$ (iii)} Let $\,\Gamma= \Gammabar \ast_T \Gammabar\,$, where $\,\Gammabar\,$ is a rayless graph of diameter $\,\,\lambda\,$ and~$T$ is a  double ray. 
As $\,T\,$ is a double ray there are exactly two adhesion sets, say~$\,S_1\,$ and $\,S_2\,$, in each copy of~$\Gammabar$. 
We define~$\hat\Gamma:=\Gammabar\setminus S_2$.
Note that~$\hat\Gamma \neq \emptyset$.
It is not hard to see that~$V(\Gamma)=\bigsqcup_{i\in\Z} V(\Gamma_i)$, where each~$\Gamma_i$ isomorphic to~$\hat\Gamma$.
We now are ready to define our quasi-isometric embedding between~$\Gamma$ and the double ray~${R=\ldots,v_1,v_0,v_1,\ldots}$.
Define~$\phi\colon V(\Gamma)\to V(R)$ such that~$\phi$ maps every vertex of~$\Gamma_i$ to the vertex~$v_i$ of~$R$.
Next we show that~$\phi$ is a quasi-isomorphic embedding.
Let~$v,v'$ be two vertices of~$\Gamma$.
We can suppose that~$v\in V(\Gamma_i)$ and~$v'\in V(\Gamma_j)$, where~$i\leq j$.
One can see that~$d_{\Gamma}(v,v')\leq (|j-i|+1)\lambda$ and so we infer that 
$$\frac{1}{\lambda} d_{\Gamma}(v,v')-\lambda\leq d_R(\phi(v),\phi(v'))=|j-i| \leq \lambda d_{\Gamma}(v,v')+\lambda.$$

As~$\phi$ is surjective we know that~$\phi$ is quasi-dense.
Thus we proved that~$\phi$ is a quasi-isometry between~$\Gamma$ and~$R$.

	\vspace{0,5cm}
	\noindent {\bf (iii)~$\Rightarrow$ (i)} Suppose that~$\phi$ is a quasi-isometry  between~$\Gamma$ and the double ray, say~$R$, with associated constant~$\lambda$.
	We shall show that~$\Gamma$ has exactly two ends, the case that~$\Gamma$ has exactly one end leads to a contradiction in an analogous manner.
	Assume to the contrary that there is a finite subset of vertices~$S$ of~$\Gamma$ such that~$\Gamma\setminus S$ has at least three big components.
	Let~$R_1:=\{u_i\}_{i\in \N}$,~${R_2:=\{v_i\}_{i\in \N}}$  and~${R_3:=\{r_i\}_{i\in \N}}$ be three rays of~$\Gamma$, exactly one in each of those big components. 
	In addition one can see that~$d_{R}(\phi(x_i),\phi(x_{i+1}))\leq 2\lambda$, where~$x_i$ and~$x_{i+1}$ are two consecutive vertices of one of those rays. 
	Since~$R$ is a double ray, we deduce that two infinite sets of~$\phi(R_i) \defi \{\phi(x)\mid x\in R_i\}$ for~$i=1,2,3$ converge to the same end of~$R$.
	Suppose that~$\phi(R_1)$ and~$\phi(R_2)$ converge to the same end.
	For a given vertex~$u_i\in R_1$ let~$v_{j_i}$ be a vertex of~$R_2$ such that the distance~$d_R(\phi(u_i),\phi(v_{j_i}))$ is minimum.
	We note that~$d_R(\phi(u_i),\phi(v_{j_i}))\leq2\lambda$.
	As~$\phi$ is a quasi-isometry we can conclude that~$d_{\Gamma}(u_i,v_{j_i})\leq3\lambda^2$. 
	Since~$S$ is finite, we can conclude that there is a vertex dominating a ray and so we have a dominated end which yields a contradiction.
\end{proof}

\begin{thm}
	\label{thinend}
	Let~$\Gamma$ be a two-ended quasi-transitive graph without dominated ends.
	Then each end of~$\Gamma$ is thin.
\end{thm}

\begin{proof}
	By Lemma~\ref{good separation} we can find a type 1  separation~$(A,A^\ast)$ of~$\Gamma$. 
	Suppose that the diameter of~${\Gamma[A\cap A^\ast]}$ is equal to~$d$.	
	Let~$C$ be a big component of~${\Gamma \setminus A\cap A^\ast}$.  
	By Lemma~\ref{far away} we can pick a vertex~$r_i$ of the ray~$R$ with distance greater than~$d$ from~$S$. 
	As~$\Gamma$ is quasi-transitive and~${A \cap A^\ast}$ contains an element from of each orbit we can find an automorphism~$g$ such that~${r_i \in g(A \cap A^\ast)}$.
	By the choice of~$r_i$ we now have that
	$$(A \cap A^\ast) \cap g(A \cap A^\ast) = \emptyset.$$ 
	Repeating this process yields a defining sequence of vertices for the end living in~$C$  each of the same finite size. 
	This implies that the degree of the end living in~$C$ is finite. 
\end{proof}	

For a two-ended quasi-transitive graph~$\Gamma$ without dominated ends let~$s(\Gamma)$ be the maximal number of disjoint double rays in~$\Gamma$. By Theorem~\ref{thinend} this is always defined. 
With a slight modification to the proof of Theorem~\ref{thinend} we obtain the following corollary:

\begin{coro}
	Let~$\Gamma$ be a two-ended quasi-transitive graphs without dominated ends.
	Then the degree of each end of~$\Gamma$ is at most~$s(\Gamma)$.
\end{coro}	

\begin{proof}
	Instead of starting the proof of Theorem \ref{thinend} with an arbitrary separation of finite order we now start with a separation~$(B,B^{\ast})$ of order~$s(\Gamma)$ separating the ends of~$\Gamma$ which we then extend to a connected separation~$(A,A^\ast)$ containing an element of each orbit. 
	The proof then follows identically with only one additional argument.
	After finding the defining sequence as images of~$(A,A^\ast)$, which is too large compared to~$s(\Gamma)$, we can reduce this back down to the separations given by the images of~$(B,B^{\ast})$ because~$(B\cap B^{\ast}) \subseteq (A \cap A^\ast)$ and because~$(B,B^{\ast})$ already separated the ends of~$\Gamma$. 
\end{proof}

It is worth mentioning that Jung \cite{jung1981note} proved that if a connected locally finite quasi-transitive graph has more than one end then it has a thin end.

\subsection{Two-ended graphs with dominated ends}
A natural question that can be raised so far is the following.
What can we say about two-ended quasi-transitive graphs with dominated ends?
An easy example could be a two-ended quasi-transitive graph with only finitely many dominating vertices in such a way that if we remove the dominating vertices, then the rest of the graph is still connected.
In this case, we discard the dominating vertices and then we apply Theorem \ref{char-two-ended-graph}.
So a strongly thin tree-amalgamation is obtained.
Now we again add the removed dominating vertices to the adhesions of the tree-amalgamation and we end up with a strongly thin tree-amalgamation for the graph. 
But examples  are not always as easy as the above the example.
In this section, we will show that we cannot expect that arbitrary two-ended quasi-transitive graphs admit a strongly thin tree-amalgamation let alone  a strongly thin tree-amalgamation satisfying the assumption of Theorem \ref{char-two-ended-graph}.
Indeed we construct a family of two-ended quasi-transitive graphs with dominated ends which do not  admit such splitting introduced in Theorem \ref{char-two-ended-graph}.
However we show that two-ended transitive graphs always admit strongly thin amalgamation.

\begin{example}
	Let~$\Gamma_1$ be an one-ended quasi-transitive graph(for instance take the complete graph~$K_{\aleph_0}$ with~$\aleph_0$ many vertices). 
	We take two copies of $\Gamma_1$ and we identify a vertex of the first copy with a vertex of the second copy.
	We call the graph by $\Gamma_{1}'$
	Take a rayless quasi-transitive graph~$\Gamma$ and join a vertex~$v$ of $\Gamma_1'$ to all vertices of~$\Gamma$.
	We obtain a new graph~$\Lambda'$ which is quasi-transitive and has exactly two ends.
\end{example}	

\begin{thm}
The graph~$\Lambda'$ does not admit any strongly thin tree-amalgamation.
\end{thm}

\begin{proof}
Assume to contrary that the graph~$\Lambda'$ admits a strongly thin tree-amalgamation~${\Lambda_1\ast_T\Lambda_2}$, where~$T$ is the double ray.
More precisely assume that~$S_i$ and~$T_i$ are adhesions of~$\Lambda_1$ and~$\Lambda_2$ in the tree-amalgamation, respectively for~$i=1,2$.
In addition let~$S_i$ correspond to~$T_i$ for~$i=1,2$.
We note that~$\Lambda_1$ and~$\Lambda_2$ are rayess graphs.
On the other hand~$\Lambda'$ is a two-ended graph.
So we can conclude that one of adhesions~$S_1$ or~$S_2$ of the tree-amalgamation separating the two ends of~$\Lambda'$ and so one of~$S_i$'s  has to contain~$v$.
Suppose that~$k$ is the maximum number of the sizes of~$S_1$ and~$S_2$ plus 1.
Since each end of~$\Lambda'$ is thick, we are able to find at least~$k$ disjoint  rays belonging to each end.
Pick one of them up and consider these~$k$ disjoint rays in the tree-amalgamation.
We note that every copy of~$\Lambda_1$ is attached to~$\Lambda_2$ via the identification map~$id\colon S_1\to T_1$ and each copy of~$\Lambda_2$ is attached to~$\Lambda_1$ via~$id\colon S_2\to T_2$.
Thus we deduce that the~$k$ disjoint rays being convergent to the end of~$\Lambda'$ meet of~$S_1(T_1)$ or~$S_2(T_2)$ and so we derive a contradiction, as the size of them is at most~$k-1$.
\end{proof}

Next we can ask ourselves what happens if we replace the condition quasi-transitivity with transitivity.
We answer to this question in the following theorem but first we need a lemma.

\begin{lemma}{\rm\cite[Propostion 4.1]{ThomassenWoess}}\label{thomasenwoess}
	Let $\Gamma$ be a connected infinite graph, let $e$ be an edge of $\Gamma$ and $k\in\mathbb N$.
	Then $\Gamma$ has finitely many $k$-tight cut meeting $e$. 
\end{lemma}	

Next we show that every two-ended transitive graph are not allowed to  have any dominated end.

\begin{thm}
Let $\Gamma$ be a two-ended graph with a dominated end such that~$\AutG$ has~$k$ orbits on~$V(\Gamma)$.
Then~$k$ is at least 2.
\end{thm}	

\begin{proof}
Assume to contrary that~$k=1$ and so~$\Gamma$ is a transitive graph.
Consider a vertex~$v\in V(\Gamma)$.
We claim that $v$ dominates both ends of $\Gamma$.
Suppose not: We can divide the vertex set into two sets.
Let $W_1$ be the set of vertices dominating one end and let $W_2$ be the rest of vertices which must dominate the other end.
We note that $W_1$ does not intersect with $W_2$.
Otherwise the intersection is not empty and the graph is transitive.
So every vertex dominates both ends.
Now take a finite separator separating two ends.
Since each vertex of the graph dominates both ends, we have a contradiction.
Hence we assume that $W_1\cap W_2=\emptyset$.
We note that $V(\Gamma)=V(W_1)\sqcup V(W_2)$.
If the number of edges $E(W_1,W_2)$ is finite, then $E(W_1,W_2)$ forms a tight cut.
Because  if $W_i$ is not connected, then a component of $W_i$ must be big for $i=1,2$ and the rest of components are small.
The small components contain dominating vertices and it yields a contradiction.
Thus $W_i$ is connected for $i=1,2$ and $E(W_1,W_2)$ forms a tight cut.
On the other hand $\Aut(\Gamma)$ is infinite and by Lemma \ref{thomasenwoess} we are able find a $g\in \Aut(\Gamma)$ such that $E(W_1,W_2)$ does not touch $gE(W_1,W_2)$ and $gE(W_1,W_2)\subseteq G[W_1]$.
Thus $gE(W_1,W_2)$ divide  $W_1$ into at least two subgraphs in such a way that one of them is small.
But each vertex of $W_1$ is dominating vertex and it yields a contradiction $|gE(W_1,W_2)|$ is finite.
Hence  $E(W_1,W_2)$ is infinite.
There is a finite separator $S$ in $\Gamma$ separating the ends.
Without of loss of generality assume that $S\subseteq W_1$.
Let $C_1$ and $C_2$ be the big components of $\Gamma\setminus S$ containing $\omega_L$ and $\omega_R$ respectively.
Furthermore we may assume that $C_2$ contains $E(W_1,W_2)$ and $\omega_L$ lives in $W_1$ and $\omega_R$ lives in $W_2$.
So there is a vertex of $W_1$ in $C_2$ and this vertex dominates the end $\omega_L$.
Therefore we derive a contradiction, as $S$ is a finite separator and infinitely many edges need to go through it to reach to $\omega_L$.
Hence the claim is proved.

As~$v$ dominates both ends of~$\Gamma$, there must be infinitely many edges crossing through any separator separating the ends that yields a contradiction.
So~$\Gamma$ cannot be a transitive graph and so~$k\geq 2$, as desired.
\end{proof}
Now the above theorem implies the following nice corollary which is the characterization of two-ended transitive graphs.

\begin{coro}
Let~$\Gamma$ be a connected transitive graphs.
Then the following statements are equivalent:
\begin{enumerate}[\rm (i)]
\item~$\Gamma$ is two-ended.
\item~$\Gamma$ can be split as a strongly thin tree-amalgamation.
\item~$\Gamma$ is quasi-isometric to the double ray. 
	\end{enumerate}
\end{coro}

\section{Groups acting on two-ended graphs}
\label{action}
In this section we investigate the action of groups on two-ended graphs without dominated ends with finitely many orbits.
We start with the following lemma which states that there are only finitely many~$k$-tight separations containing a given vertex. 
Lemma~\ref{New Lemma 5} is a separation version of a result of Thomassen and Woess for vertex cuts~\cite[Proposition 4.2]{ThomassenWoess} with a proof which is quite closely related to their proof.

\begin{lemma}\label{New Lemma 5}
	Let~$\Gamma$ be a two-ended graph without dominated ends then for any vertex~$v \in V(\Gamma)$ there are only finitely many~$k$-tight separations containing~$v$.
\end{lemma}

\begin{proof}
	We apply induction on~$k$. 
	The case~$k =1$ is trivial.
	So let~$k \geq 2$ and let~$v$ be a vertex contained in the separator of a~$k$-tight separation~$(A,A^\ast)$.
	Let~$C_1$ and~$C_2$ be the two big components of~$\Gamma \setminus (A \cap A^\ast)$.
	As~$(A,A^\ast)$ is a~$k$-tight separation we know that~$v$ is adjacent to both~$C_1$ and~$C_2$.
	We now consider the graph~$\Gamma^- \defi \Gamma -v$. 
	As~$v$ is not dominating any ends we can find a finite vertex set~$S_1 \subsetneq C_1$ and~$S_2 \subsetneq C_2$ such that~$S_i$ separates~$v$ from the end living in~$C_i$ for~$i \in \{1,2\}$.\footnote{A finite vertex set~$S$ separates a vertex~$v \notin S$ from an end~$\w$ if~$v$ is not contained in the component~$G \setminus S$ which~$\w$ lives.}
	For each pair~$x,y$ of vertices with~$x \in S_1$ and~$y \in S_2$ we now pick a~$x-y$ path~$P_{xy}$ in~$\Gamma^-$.
	This is possible as~$k \geq 2$ and because~$(A,A^\ast)$ is~$k$-tight. 
	Let~$\Pcal$ be the set of all those paths and let~$V_P$ be the set of vertices contained in the path contained in~$\Pcal$.
	Note that~$V_P$ is finite because each path~$P_{xy}$ is finite and both~$S_1$ and~$S_2$ are finite. 
	By the hypothesis of the induction we know that for each vertex in~$V_P$ there are only finitely~$(k-1)$-tight separations meeting that vertex. 
	So we infer that there are only finitely many~$(k-1)$-tight separations of~$\Gamma^-$ meeting~$V_P$.
	Suppose that there is a~$k$-tight separation~$(B,\B)$ such that~$v \in B \cap \B$  and~$B \cap \B$ does not meet~$V_P$. 
	As~$(B,\B)$ is~$k$-tight we know that~$v$ is adjacent to both big components of~$\Gamma \setminus B \cap \B$. 
	But this contradicts our choice of~$S_i$. 
	Hence there are only finitely many~$k$-tight separations containing~$v$, as desired.
\end{proof}

In the following we extend the notation of diameter from connected graphs to not necessarily connected graphs. 
Let~$\Gamma$ be a graph.
We denote the set of all subgraphs of~$\Gamma$ by~$\mathcal P(\Gamma)$.
We define the function~$\rho\colon\mathcal P(\Gamma)\to\Z \cup \{\infty\}$ by setting~$\rho(X)=\mathsf{sup}\{\mathsf{diam}(C)\mid C\text{ is a component of } X\}$.\footnote{If the component~$C$ does not have finite diameter, we say its diameter is infinite.}

\begin{lemma}\label{finitecomplement}
	Let~$\Gamma$ be a quasi-transitive two-ended graph without dominated ends such that  $|\mathsf{St}(v)|<\infty$ for every vertex $v$ of $\Gamma$ and let~$(A,A^*)$ be a tight separation of~$\Gamma$.
	Then for infinitely many~${g\in \mathsf{Aut}(\Gamma)}$ either the number~$\rho({A\Delta gA})$ or~$\rho(A\Delta gA)^c$ is finite.
\end{lemma}

\begin{proof}
	It follows from Lemma \ref{New Lemma 5} and $|\Gamma_v|<\infty$ that~$(A,A^\ast)$ and~$g(A,A^\ast)$ are nested for all but finitely many~$g\in\AutG$.
	Let~$g \in \AutG$ such that
	$${(A\cap A^\ast) \cap g(A \cap A^\ast) = \emptyset}.$$
	By definition we know that either~${A\Delta gA}$ or~$({A\Delta gA})^c$ contains a ray.
	Without loss of generality we may assume the second case. 
	The other case is analogous. 
	We now show that the number~${\rho(A\Delta gA)}$ is finite. 
	Suppose that~$C_1$ is the big component of~$\Gamma\setminus (A\cap A^\ast)$ which does not meet~$g(A\cap A^\ast)$ and~$C_2$ is the big component of~$\Gamma\setminus g(A\cap A^\ast)$ which does not meet~$(A\cap A^\ast)$.
	By Lemma~\ref{far away} we are able to find type 1 separations~$(B,\B)$ and~$(C,\C)$ in such a way that~${B\cap\B\subsetneq C_1}$ and~${C\cap \C\subsetneq C_2}$ and such that the~$B \cap \B$ and~$C \cap \C$ each have empty intersection with~$A \cap \A$ and~$g(A \cap \A)$. 
	Now it is straightforward to verify that~${A\Delta gA}$ is contained in a rayless component~$X$ of~${\Gamma \setminus \left( (B \cap \B)\bigcup (C \cap\C)\right )}$.
	Using Lemma~\ref{smalldiameter} we can conclude that~$X$ has finite diameter and hence~${\rho(A \Delta gA)}$ is finite.
\end{proof}

Assume that an infinite group~$G$ acts on a two-ended graph~$\Gamma$ without dominated ends with finitely many orbits and let~$(A,A^\ast)$ be a tight separation of~$\Gamma$.
By Lemma~\ref{finitecomplement} we may assume~$\rho(A\Delta gA)$ is finite.
We set
~$$H:=\{g\in G\mid \rho(A\Delta gA) < \infty \}.$$
We call~$H$ the \emph{separation subgroup} induced by~$(A,A^\ast)$.\footnote{See the proof of Lemma~\ref{index 2 subgroup} for a proof that~$H$ is indeed a subgroup.}
In the sequel we study separations subgroups.
We note that we infer from Lemma \ref{finitecomplement} that~$H$ is infinite.

\begin{lemma}\label{index 2 subgroup}
	Let~$G$ be  an infinite group acting on a two-ended graph~$\Gamma$ without dominated ends with finitely many orbits almost freely. 
	Let~$H$ be the separation subgroup induced by a tight separation~$(A,A^\ast)$ of~$\Gamma$.
	Then~$H$ is a subgroup of~$G$ of index at most~$2$.
\end{lemma}	

\begin{proof}
	We first show that~$H$ is indeed a subgroup of~$G$. 
	As automorphisms preserve distances it is that for~$h \in H,g \in G$ we have 
	$$\rho (g(A\Delta hA))=\rho (A\Delta hA)<\infty.$$ 
	As this is in particular true for~$g = h^{-1}$ we only need to show that~$H$ is closed under multiplication and this is straightforward to check as one may see that 
	\begin{align*}
	A\Delta h_1h_2A& =(A\Delta h_1A)\Delta (h_1A\Delta h_1h_2A)\\
	& =(A\Delta h_1A)\Delta h_1(A\Delta h_2A) .
	\end{align*}
	Since~$\rho(A\Delta h_iA)$ is finite for~$i=1,2$, we conclude  that~$h_1h_2$ belongs to~$H$.
	
	Now we only need to establish that~$H$ has index at most two in~$G$.
	Assume that~$H$ is a proper subgroup of~$G$ and that the index of~$H$ is bigger than two.
	Let~$H$ and~$Hg_i$ be three distinct cosets for~${i=1,2}$.
	Furthermore by Lemma~\ref{finitecomplement} we may assume~${\rho((A \Delta g_i A})^c)$ is finite for~$i=1,2$ . 
	Note that
	$$A\Delta g_1g^{-1}_2A=(A\Delta g_1A)\Delta g_1(A\Delta g_2^{-1}A).$$
	On the other hand we already know that
	$$A\Delta g_1g^{-1}_2A=(A\Delta g_1A)^c\Delta (g_1(A\Delta g_2^{-1}A))^c.$$ 
	We notice that the diameter of~$A\Delta g_i A$ is infinite for~$i=1,2$. 
	Since~$g_2 \notin H$ we know that~$g_2^{-1} \notin H$ and so~$\rho(g_1(A\Delta g_2^{-1}A))$ is infinite. 
	By Lemma~\ref{finitecomplement} we infer that~$\rho(g_1(A\Delta g_2^{-1}A)^c)$ is finite. 
	Now as the two numbers~${\rho((A\Delta g_1A)^c)}$ and~${\rho( g_1(A\Delta g_2^{-1}A)^c)}$ are finite we conclude that~${\rho A \Delta g_1g_2^{-1}A < \infty}$.
	Thus we conclude that~$g_1g_2^{-1}$ belongs to~$H$.
	It follows that~${H = H g_1g_2^{-1}}$ and multiplying by~$g_2$ yields~${Hg_1 = Hg_2}$ which contradicts~$H g_1 \neq H g_2$. 
\end{proof}	 

We now are ready to state the main theorem of this section.

\begin{theorem}\label{cyclic finite index}
	Let~$G$ be a group acting with only finitely many orbits on a two-ended graph~$\Gamma$ without dominated ends almost freely.
	Then~$G$ contains an infinite cyclic subgroup of finite index.
\end{theorem}

\begin{proof} 
	Let~$(A,A^*)$ be a tight separation and let~$(\bar{A}, \bar{A}^\ast)$ be the type 2 separation given by Corollary~\ref{type 2 separation}.
	Additionally let~$H$ be the separation subgroup induced by~$(A,A^\ast)$. 
	We now use Lemma~\ref{element infinite order} on~$(\bar{A},\bar{A}^*)$ to find an element~${h \in G}$ of infinite order.
	It is straightforward to check that~$h \in H$. 
	Now it only remains to show that~$L \defi \langle h \rangle$ has finite index in~$H$. 
	
	Suppose for a contradiction that~$L$ has infinite index in~$H$ and for simplicity set~$Z:= A\cap A^*$.
	This implies that~$H = \bigsqcup_{i \in \bbN} L h_i$. 
	We have the two following cases:\\
	{\bf{Case I:}} There are infinitely many~${i\in \N}$ and~${j_i\in \N}$  such that~${h_iZ=h^{j_i}Z}$ and so~${Z=h^{-j_i}h_iZ}$.
	It follows from Lemma \ref{New Lemma 5} that there are only finitely many~$f$-tight separations meeting~$Z$ where~$|Z|=f$.
	We infer that there are infinitely many~$k\in \N$ such that~$h^{-j_\ell}h_{\ell}Z=h^{-j_k}h_kZ$ for a specific~$\ell\in\N$.
	Since the size of~$Z$ is finite, we deduce that there is a~$v\in Z$ such that for a specific~$m\in \N$ we have~$h^{-j_m}h_{m}v=h^{-j_n}h_nv$ for infinitely many~$n\in\N$.
	So we are able to conclude that the stabilizer of~$v$ is infinite which is a contradiction.
	Hence for~$n_i\in \N$ where~$i=1,2$ we have to have 
	$$(h^{-j_m}h_{m}^{-1})h^{-j_{n_1}}h_{n_1}=({h^{-j_m}}h_m)^{-1}h^{-j_{n_2}}h_{n_2}.$$
	The above equality implies that~$Lh_{n_1}=Lh_{n_2}$ which yields a contradiction.
	\newline
	{\bf{Case II:}} We suppose that there are only finitely many~${i\in \N}$ and~${j_i\in \N}$  such that~$h_iZ=h^{j_i}Z$.
	We are going to define the graph~$X \defi \Gamma[A \Delta hA]$ and we conclude that~${\Gamma=\cup_{i\in\Z} h^iX}$.
	We can assume that~${h_iZ\subseteq h^{j_i}X}$ for infinitely many~$i\in N$ and~$j_i\in \N$ and so we have~$h^{-j_i}h_iZ\subseteq X$.
	Let~$p$ be a shortest path between~$Z$ and~$hZ$.
	For every vertex~$v$ of~$p$, by Lemma \ref{New Lemma 5} we know that there are finitely many tight separation~$gZ$ for~$g\in G$ meeting~$v$.
	So we infer that there are infinitely many~$k\in \N$ such that~$h^{-j_\ell}h_{\ell}Z=h^{-j_k}h_kZ$ for a specific~$\ell\in\N$.
	Then with an analogue method we used for the preceding case, we are able to show that the stabilizer of at least one vertex of~$Z$ is infinite and again we conclude that~$(h^{-j_m}h_{m}^{-1})h^{-j_{n_1}}h_{n_1}=({h^{-j_m}}h_m)^{-1}h^{-j_{n_2}}h_{n_2}$ for~$n_1,n_2\in \N$.
	Again it yields a contradiction. Hence each case gives us a contradiction and it proves our theorem as desired.
                \end{proof}

We close the paper with the following corollary which is an immediate consequence of the above theorem and Theorem \ref{Classifiication}. 

\begin{coro}
Let~$G$ be an infinite group acting with only finitely many orbits on a two-ended graph~$\Gamma$ without dominated ends almost freely.
Then~$G$ is two-ended.\qed
\end{coro}


\bibliographystyle{plain}
\bibliography{collective.bib}

   \end{document}